\RequirePackage{rotating}
\documentclass{amsart}

\usepackage{amsmath}
\usepackage{amssymb}
\usepackage{amsthm}
\usepackage{url}
\usepackage{xypic}
\usepackage{rotating}
\usepackage{mathdots}

\usepackage{changepage}

\usepackage{hyperref}

\newtheorem{theorem}{Theorem}[section]
\newtheorem{lemma}[theorem]{Lemma}

\newtheorem{proposition}[theorem]{Proposition}
\newtheorem{claim}[theorem]{Claim}

\theoremstyle{definition}
\newtheorem{definition}[theorem]{Definition}

\newtheorem{question}[theorem]{Question}
\newcounter{cases}[theorem]
\newtheorem{case}[cases]{Case}

\theoremstyle{remark}

\DeclareMathOperator{\tp}{tp}

\newcommand{\mc}[1]{\mathcal{#1}}

\newcommand{\concat}[0]{\textrm{\^{}}}
\newcommand{\la}[0]{\langle}
\newcommand{\ra}[0]{\rangle}

\begin{document}

\title[Scott sentences of f.g.~structures]{On optimal Scott sentences of finitely generated algebraic structures}

\author{Matthew Harrison-Trainor}
\address{Group in Logic and the Methodology of Science\\
University of California, Berkeley\\
 USA}
\email{matthew.h-t@berkeley.edu}
\urladdr{\href{http://www.math.berkeley.edu/~mattht/index.html}{www.math.berkeley.edu/$\sim$mattht}}

\author{Meng-Che Ho}
\address{Department of Mathematics\\
University of Wisconsin--Madison\\
USA}
\email{ho@math.wisc.edu}
\urladdr{\href{http://www.math.wisc.edu/~ho/}{www.math.wisc.edu/$\sim$ho/}}

\begin{abstract}
Scott showed that for every countable structure $\mathcal{A}$, there is a sentence of the infinitary logic $\mc{L}_{\omega_1\omega}$, called a Scott sentence for $\mc{A}$, whose models are exactly the isomorphic copies of $\mathcal{A}$. Thus, the least quantifier complexity of a Scott sentence of a structure is an invariant that measures the complexity ``describing" the structure. Knight et al.~have studied the Scott sentences of many structures. In particular, Knight and Saraph showed that a finitely generated structure always has a $\Sigma^0_3$ Scott sentence. We give a characterization of the finitely generated structures for whom the $\Sigma^0_3$ Scott sentence is optimal. One application of this result is to give a construction of a finitely generated group where the $\Sigma^0_3$ Scott sentence is optimal.
\end{abstract}

\maketitle

\section{Introduction}

Given a countable structure $\mc{M}$, we can describe $\mc{M}$, up to isomorphism, by a sentence of the infinitary logic $\mc{L}_{\omega_1 \omega}$ which allows countable conjunctions and disjunctions. (See Section \ref{scott sentences} for the formal description of this logic; in this brief introduction, we will write down sentences in an informal way.) To measure the complexity of $\mc{M}$, we want to write down the simplest possible description of $\mc{M}$. For example, one can describe the countably infinite-dimensional $\mathbb{Q}$-vector space by the vector space axioms together with the sentence
\smallskip{}
\begin{adjustwidth}{1em}{0pt}
\textit{for all $n$, there are $x_1,\ldots,x_n$ such that for all $r_1,\ldots,r_n \in \mathbb{Q}$, if $r_1 x_1 + \cdots + r_n x_n = 0$ then some $r_i = 0$.}
\end{adjustwidth}
\smallskip{}
This sentence has a universal quantifier, followed by an existential quantifier, followed by a universal quantifier. There is a hierarchy of sentences depending on the number of quantifier alternations. The $\Sigma^0_n$ sentences have $n$ alternations of quantifiers, beginning with existential quantifiers; the $\Pi^0_n$ sentences have $n$ alternations of quantifiers, beginning with a universal quantifier; and the d-$\Sigma^0_n$ sentences are the conjunction of a $\Sigma^0_n$ and a $\Pi^0_n$ sentence. The hierarchy is ordered as follows, from the simplest formulas on the left, to the most complicated formulas on the right:
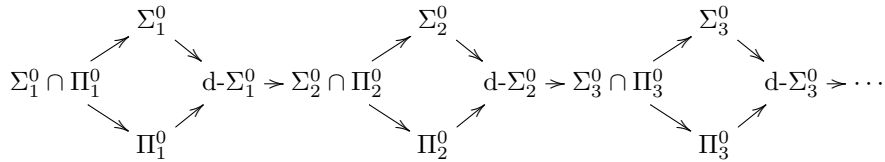
\[
\xymatrix@=7pt{
                                                      & \Sigma^0_1 \ar[dr]&                                      &                                                      &\Sigma_2^0\ar[dr]&&& \Sigma^0_3 \ar[dr] & & \\
\Sigma^0_1\cap\Pi^0_1 \ar[ur]\ar[dr] &                            &\text{d-}\Sigma^0_1 \ar[r]&\Sigma^0_2\cap\Pi^0_2 \ar[ur]\ar[dr]&&\text{d-}\Sigma^0_2 \ar[r]&\Sigma^0_{3}\cap\Pi^0_{3} \ar[ur] \ar[dr] & & \text{d-}\Sigma^0_3 \ar[r] & \cdots\\
                                                      & \Pi^0_1 \ar[ur]      &                                      &                                                     &\Pi_2^0\ar[ur]&&& \Pi^0_3    \ar[ur] & &
}
\]
\noindent We use this hierarchy to measure the complexity of a sentence. The sentence given above describing the infinite-dimensional $\mathbb{Q}$-vector space is a $\Pi^0_3$ sentence, and it turns out that this is the best possible; there is no d-$\Sigma^0_2$ description of this vector space. There is a d-$\Sigma^0_2$ description of any finite-dimensional $\mathbb{Q}$-vector space, and so these structures are ``simpler'' than the infinite-dimensional vectors space.

In this paper, we consider descriptions of finitely generated structures, and particularly of finitely generated groups. Any finitely generated structure $\mc{M}$, with generating tuple $\bar{a}$, has a $\Sigma^0_3$ description of the form:
\smallskip{}
\begin{adjustwidth}{1em}{0pt}
\textit{there is a tuple $\bar{x}$, satisfying the same atomic formulas as $\bar{a}$ (i.e., for all atomic formulas true of $\bar{a}$, the formula is true of $\bar{x}$), such that every element is generated by $\bar{x}$ (i.e., for all $y$, there is a term $t$ in the language such that $y = t(\bar{x})$).}
\end{adjustwidth}
\smallskip{}
However, many finitely generated groups have a simpler description which is d-$\Sigma^0_2$. For the group $\mathbb{Z}$, for example, the $\Pi^0_2$ axioms of torsion-free abelian groups, together with the following two sentences, which are $\Pi^0_2$ and $\Sigma^0_2$ respectively, form a d-$\Sigma^0_2$ description:
\smallskip{}
\begin{adjustwidth}{1em}{0pt}
\textit{for all $x$ and $y$, there are $n,m \in \mathbb{Z}$, not both zero, such that $n x = m y$}
\end{adjustwidth}
\smallskip{}
and
\smallskip{}
\begin{adjustwidth}{1em}{0pt}
\textit{there is $x \neq 0$ which has no proper divisors.}
\end{adjustwidth}
\smallskip{}
Indeed, all previously known examples of finitely generated groups had a d-$\Sigma^0_2$ Scott sentences, including all polycyclic (including nilpotent) groups and many finitely-generated solvable groups \cite{Ho}. The main result of this paper is an example of a computable group which has no d-$\Sigma^0_2$ Scott sentence. Our group has $\Sigma^0_3$ $m$-complete index set.

This paper is divided into two main sets of results. The first is a general investigation of conditions for a finitely-generated structure to have (or not have) a d-$\Sigma^0_2$ Scott sentence. The second is an application of these general results to constructing the group mentioned above. We also include some results on finitely generated fields and rings.

\subsection{Scott sentences} \label{scott sentences}

The infinitary logic $\mc{L}_{\omega_1 \omega}$ is the logic which allows countably infinite conjunctions and disjunctions but only finite quantification. If the conjunctions and disjunctions of a formula $\varphi$ are all over computable sets of indices for formulas, then we say that $\varphi$ is computable. 

We use the following recursive definition to define the complexity of classes:
\begin{itemize}
\item An $\mc{L}_{\omega_1 \omega}$ formula is both $\Sigma_0^0$ and $\Pi_0^0$ if it is quantifier free and does not contain any infinite disjunction or conjunction. 
\item An $\mc{L}_{\omega_1 \omega}$ formula is $\Sigma_\alpha^0$ if it is a countable disjunction of formulas of the form $\exists x \phi$ where each $\phi$ is $\Pi_\beta^0$ for some $\beta < \alpha$.
\item An $\mc{L}_{\omega_1 \omega}$ formula is $\Pi_\alpha^0$ if it is a countable disjunction of formulas of the form $\forall x \phi$ where each $\phi$ is $\Sigma_\beta^0$ for some $\beta < \alpha$.
\end{itemize}
We say a formula is $\text{d-}\Sigma_\alpha^0$ if it is a conjunction of a $\Sigma_\alpha^0$ formula and a $\Pi_\alpha^0$ formula. 

Scott \cite{Scott65} showed that if $\mc{A}$ is a countable structure in a countable language, then there is a sentence $\varphi$ of $\mc{L}_{\omega_1 \omega}$ whose countable models are exactly the isomorphic copies of $\mc{A}$. Such a sentence is called a \textit{Scott sentence} for $\mc{A}$. We remark that because $\alpha \wedge \neg (\beta \wedge \neg \gamma)$ is equivalent to $(\alpha \wedge \neg \beta) \vee (\alpha \wedge \gamma)$, the complexity classes $\text{$n$-}\Sigma^0_\alpha$ of Scott sentences collapse for $n \geq 2$.

We can measure the complexity of a countable structure by looking for a Scott sentence of minimal complexity, as measured by the quantifier complexity hierarchy of computable formulas described above. \cite{D.E.Miller78} showed that if $\mc{A}$ has a $\Pi^0_{\alpha}$ Scott sentence and a $\Sigma^0_{\alpha}$ Scott sentence, then it must have a $\text{d-}\Sigma^0_\beta$ Scott sentence for some $\beta < \alpha$. So for a given structure, the optimal Scott sentence is $\Sigma^0_\alpha$, $\Pi^0_\alpha$, or d-$\Sigma^0_\alpha$ for some $\alpha$.

We refer the interested readers to Chapter 6 of \cite{AshKnight00} for a more complete description of $\mc{L}_{\omega_1 \omega}$ formulas and Scott sentences. 

\subsection{Index set complexity}

Given a structure $\mc{A}$ and a Scott sentence $\varphi$ for $\mc{A}$, we want to determine whether $\varphi$ is an optimal Scott sentence for $\mc{A}$, or whether there is a simpler Scott sentence which we have not yet found. We can use index set calculations to resolve this problem.

\begin{definition}
Let $\mc{A}$ be a structure. The index set $I(\mc{A})$ is the set of all indices $e$ such that the $e$th Turing machine $\Phi_e$ gives the atomic diagram of a structure $\mc{B}$ isomorphic to $\mc{A}$. We can also relativize this to any set $X$: $I^X(\mc{A})$ is the set of all indices $e$ such that the $e$th Turing machine $\Phi^X_e$ with oracle $X$ gives the atomic diagram of a structure $\mc{B}$ isomorphic to $\mc{A}$.
\end{definition}

\noindent There is a connection between index sets and Scott sentences:

\begin{proposition}
If a countable structure $\mc{A}$ has an $X$-computable $\Sigma^0_\alpha$ (respectively $\Pi^0_\alpha$ or \textup{d-}$\Sigma^0_\alpha$) Scott sentence, then the index set $I^X(\mc{A})$ is in $\Sigma^0_\alpha(X)$ (respectively $\Pi^0_\alpha(X)$ or \textup{d-}$\Sigma^0_\alpha(X)$).
\end{proposition}

\noindent So if, for example, we have a computable $\Sigma^0_3$ Scott sentence for a structure $\mc{A}$, we will try to show that the index set $I(\mc{A})$ is $\Sigma^0_3$ $m$-complete. If we can do this, then we know that our Scott sentence is optimal. In general, any $\mc{L}_{\omega_1 \omega}$ sentence is $X$-computable for some $X$.

\subsection{Summary of prior results}

There are many results using the strategy above to find the complexities of optimal Scott sentences of structures. For example, Knight et al.~\cite{Ca06}, \cite{Ca12} determined the complexities of optimal Scott sentences for finitely generated free abelian groups, reduced abelian groups, free groups, and many other structures. 

However, this strategy does not work when the complexity of the optimal Scott sentence is strictly higher than the complexity of the index set. Indeed, Knight and McCoy gave the first such example in \cite{Kn14}, showing there is a subgroup $G$ of $\mathbb{Q}$ such that $I(G)$ is d-$\Sigma^0_2$, but it has no computable d-$\Sigma^0_2$ Scott sentence.

It was observed in \cite{Kn16} that any computable finitely generated group, and indeed any computable finitely generated structure, has a computable $\Sigma^0_3$ Scott sentence. In \cite{Ho}, it was shown that many classes of ``nice'' groups in the sense of geometric group theory, including polycyclic groups (which includes nilpotent groups and abelian groups), and certain solvable groups all have computable d-$\Sigma^0_2$ Scott sentence. However, none of these examples achieves the $\Sigma^0_3$ bound that was given in \cite{Kn16}.

\subsection{New results}

In this paper, we give an example of a finitely-generated group which has no $\text{d-}\Sigma^0_2$ Scott sentence. As mentioned above, we do this by showing that the index set is $\Sigma^0_3$ $m$-complete.

\begin{theorem}
There is a finitely-generated computable group $G$ whose index set is $\Sigma^0_3$ $m$-complete.
\end{theorem}

\noindent The proof is in two parts. First, in Section \ref{general theory}, we develop some general results on when a finitely generated structure of any kind has a $\text{d-}\Sigma^0_2$ Scott sentence. These results are of interest independent of their application to groups.

\begin{definition}
Let $\mc{A}$ be a finitely generated structure. Then $\mc{A}$ is \textit{self-reflective} if it contains a proper $\Sigma^0_1$-elementary substructure isomorphic to itself. ($\mc{B}$ is a $\Sigma^0_1$-elementary substructure of $\mc{A}$, and we write $\mc{B} \preceq_1 \mc{A}$, if, for each existential formula $\varphi(\bar{x})$ and $\bar{b} \in \mc{B}$, $\mc{A} \models \varphi(\bar{b})$ if and only if $\mc{B} \models \varphi(\bar{b})$).
\end{definition}

\noindent We prove, using an index-set calculation, the equivalence of (1) and (2) in the following characterization of finitely-generated structures with no d-$\Sigma^0_2$ Scott sentence.
\begin{theorem}\label{main theorem}
Let $\mc{M}$ be a finitely generated structure. The following are equivalent:
\begin{enumerate}
\item $\mc{M}$ has a \textup{d-}$\Sigma^0_2$ Scott sentence,
\item $\mc{M}$ is not self-reflective,
\item for all (or some) generating tuples of $\mc{M}$, the orbit is defined by a $\Pi^0_1$ formula.
\end{enumerate}
\end{theorem}

\noindent The equivalence of (3) to (1) has been proved by Alvir, Knight, and McCoy \cite{AlvirKnightMcCoy}.

Second, in Section \ref{finitely generated groups}, we apply this characterization to finitely generated groups. Using small cancellation theory and HNN extensions, we produce a computable group $G$ which is self-reflective. Thus---using Theorem \ref{main theorem}---this group has no d-$\Sigma^0_2$ Scott sentence. Using the group ring construction, we generalize this in Section \ref{finitely generated rings} to produce a ring which is self-reflective.

We also apply our results to finitely generated fields in Section \ref{finitely generated fields}. A simple argument shows that no finitely generated field is self-reflective. Thus:

\begin{theorem}\label{thm:field}
Every finitely generated field has a \textup{d-}$\Sigma^0_2$ Scott sentence.
\end{theorem}

\subsection{Open questions}

We leave here several open questions. First, a special class of finitely generated groups are the finitely presented groups. Is there a (computable) finitely presented group with no d-$\Sigma^0_2$ Scott sentence?

\begin{question}
Does every finitely presented group with solvable word problem have a d-$\Sigma^0_2$ Scott sentence?
\end{question}

Second, one can consider structures other than fields and groups. A natural class to consider is rings. Using the group ring construction, we get a self-reflective ring.
However, if we insist that the ring be commutative, then such a construction no longer works.

\begin{question}
Does every commutative ring have a d-$\Sigma^0_2$ Scott sentence?
\end{question}

\noindent One can also place further restrictions on the ring. A natural restriction is that there be no zero-divisors.

\begin{question}
Does every integral domain have a d-$\Sigma^0_2$ Scott sentence?
\end{question}

\noindent We expect the answer to be yes, as integral domains have a good dimension theory.

\section{General theory}\label{general theory}

Our goal in this section is to prove Theorem \ref{main theorem}. The proof is in two parts. First we will show that if $\mc{A}$ is not self-reflective, then it has a d-$\Sigma^0_2$ Scott sentence. Second, we will show that if $\mc{A}$ is self-reflective, then its index set is as complicated as possible.

\begin{theorem}\label{thm:not-refl}
Let $\mc{A}$ be a finitely generated structure. If $\mc{A}$ is not self-reflective, then $\mc{A}$ has a \textup{d-}$\Sigma^0_2$ Scott sentence.
\end{theorem}
\begin{proof}
Let $\bar{g}$ be a generating tuple for $\mc{A}$. Let $p$ be the atomic type of $\bar{g}$. For any tuple $\bar{g}'$ satisfying $p$, the substructure generated by $\bar{g}$ is isomorphic to $\mc{A}$. Since $\mc{A}$ is not self-reflective, if $\bar{g}'$ does not generate $\mc{A}$, then there is a tuple $\bar{a}$ and a quantifier-free formula $\psi(\bar{x},\bar{y})$ with $\mc{A} \models \psi(\bar{g}',\bar{a})$, such that there is no $\bar{b} \in \mc{A}$ such that $\mc{A} \models \psi(\bar{g},\bar{b})$. Let $S$ be the set of formulas $\psi(\bar{x},\bar{y})$ such that for some tuple $\bar{g}'$ satisfying the atomic type $p$ but not generating $\mc{A}$, and some $\bar{a}$, $\mc{A} \models \psi(\bar{g}',\bar{a})$, but there is no $\bar{b} \in \mc{A}$ such that $\mc{A} \models \psi(\bar{g},\bar{b})$.

Using the set $S$, we can now define the Scott sentence for $\mc{A}$. The Scott sentence for $\mc{A}$ is the conjunction of the $\Sigma^0_2$ sentence which says:
\smallskip{}
\begin{adjustwidth}{1em}{0pt}
\textit{there exists a tuple $\bar{x}$ satisfying $p$ and such that for all $\bar{z}$ and $\psi \in S$, $\bar{x}\bar{z}$ does not satisfy $\psi$,}
\end{adjustwidth}
\smallskip{}
and the $\Pi^0_2$ sentence which says: 
\smallskip{}
\begin{adjustwidth}{1em}{0pt}
\textit{for all tuples $\bar{x}$ which satisfy $p$, either for all $y$, $y \in \la \bar{x} \ra$, or there is a formula $\psi \in S$ and a tuple $\bar{z}$ such that $\bar{x},\bar{z}$ satisfies $\psi$.}
\end{adjustwidth}
\smallskip{}
This latter sentence is of the form $(\forall \bar{x})\left[\theta \rightarrow (\alpha \vee \beta)\right]$ where $\theta$ is $\Pi^0_1$, $\alpha$ is $\Pi^0_2$, and $\beta$ is $\Sigma^0_1$.

It is easy to see that $\mc{A}$ models this sentence. Now suppose that $\mc{M}$ is any structure which satisfies this sentence. Since $\mc{M}$ satisfies the $\Sigma^0_2$ part of the sentence, there is a tuple $\bar{h} \in \mc{M}$ which satisfies the atomic type $p$, and such that for all $\bar{c} \in \mc{M}$ and $\psi \in S$, $\mc{M} \nvDash \psi(\bar{h},\bar{c})$. We claim that $\bar{h}$ generates $\mc{M}$; since $\bar{h}$ satisfies the atomic type $p$, this would imply that $\mc{M}$ is isomorphic to $\mc{A}$. Indeed, by the $\Pi^0_2$ part of the sentence, either $\bar{h}$ generates $\mc{M}$ or there is a formula $\psi \in S$ and a tuple $\bar{c}$ such that $\mc{M} \models \psi(\bar{h},\bar{c})$. The latter cannot happen, and so $\bar{h}$ generates $\mc{M}$. 
\end{proof}

We will now show that if $\mc{A}$ is self-reflective, then (relativizing everything to $\mc{A}$) its index set is $\Sigma^0_3$ $m$-complete. We will use the following remark in the proof.

\begin{theorem}\label{completeness}
Let $\mc{A}$ be $X$-computable and self-reflective. Then $I^{X}(\mc{A})$ is $\Sigma^0_3(X)$ $m$-complete (relative to $X$).
\end{theorem}
\begin{proof}
We will assume that $\mc{A}$ is computable; the general result can be obtained by relativizing. Fix a $\Sigma^0_3$ set $S$. We may assume that $S$ is of the form
\[ n \in S \Longleftrightarrow (\exists e)\, \text{$W_{f(e,n)}$ is infinite}\]
for some computable function $f$.
We will define, uniformly in $n$, a computable structure $\mc{B}_n$ such that if $n \in S$, then $\mc{B}_n \cong \mc{A}$, and if $n \notin S$, then $\mc{B}_n$ is not finitely generated. We may assume that at each stage $s$, there is at most one $e$ for which an element is enumerated into $W_{f(e,n)}$.

For convenience, we will suppress $n$, writing $\mc{B}$ for $\mc{B}_n$ and $f(e)$ for $f(e,n)$. We will build $\mc{B}$ with domain $\omega$ as a union of finite substructures (in a finite sublanguage) $\mc{B}[s]$, viewing the language as a relation language as is usual for this kind of construction.

Since $\mc{A}$ sits properly inside itself as a $\Sigma^0_1$-elementary substructure, we can create an infinite chain
\[ \mc{A}_0 \prec_1 \mc{A}_1 \prec_1 \mc{A}_2 \prec_1 \cdots \prec_1 \mc{A}^* \]
where each $\mc{A}_i$ is (effectively) isomorphic to $\mc{A}$ and $\mc{A}_i$ is a c.e.\ (but not necessarily computable) subset of $\mc{A}_{i+1}$. The structure $\mc{A}^*$ is the union of all of the $\mc{A}_i$'s, and is not finitely generated (and hence not isomorphic to $\mc{A}$).

At each stage $s$, the domain of $\mc{B}[s]$ will be the union of finitely many unary relations $R_0[s] \subseteq \cdots \subseteq R_{k_s}[s]$. We will also have computable partial embeddings $j[s] \colon \mc{B}[s] \to \mc{A}^*$ such that $j[s](R_k[s]) \subseteq \mc{A}_k$.

We will build $R_0$ isomorphic to $\mc{A}_0$, $R_1$ isomorphic to $\mc{A}_1$, and so on, via $j$. While $W_{f(e)}$ does not have any elements enumerated into it, we will keep building $R_e$ to copy $\mc{A}_e$. However, when an element is enumerated into $W_{f(e)}$ we will collapse each $R_j$, $j > e$ into $R_{e}$. If $e$ is least such that $W_{f(e)}$ is infinite, then $\mc{B}$ will consist just of the domain $R_{e}$, as each $R_j$, $j > e$, will be collapsed infinitely many times, and $\mc{B}$ will be isomorphic to $\mc{A}$. On the other hand, if each $W_{f(e)}$ is finite, then $\mc{B}$ will be isomorphic to $\mc{A}^*$, and hence $\mc{B}$ will not be isomorphic to $\mc{A}$.

\bigskip{}

\noindent \textit{Construction.} Begin at stage $0$ with $\mc{B}[0]$ empty and $k_0 = 0$, with $R_0[0]$ empty.

\medskip{}

\noindent \textit{Action at stage $s+1 = 3t+1$.} Set $k = k_s$. We will have $k_{s+1} = k$. For each $n = 0,\ldots,k$, let $a_n$ be the first element of $\mc{A}_n$ not in $j[s](R_n[s])$. Define $\mc{B}[s+1] \supseteq \mc{B}[s]$ so that $j[s+1] \colon \mc{B}[s+1] \to \mc{A}^*$ is a partial embedding, extending $j[s]$, whose range also contains $a_0,\ldots,a_k$. Given $x \in \mc{B}[s+1]$, set $R_n[s+1]$ to be $R_n[s]$ plus the elements $x$ such that $j(x)$ is among the first $s$ elements of $\mc{A}_n$.

\medskip{}

\noindent \textit{Action at stage $s+1 = 3t+2$.} Set $k_{s+1} = k_s + 1$ and $j[s+1]= j[s]$. Let $R_{k_{s+1}}$ be empty. For each $n = 0,\ldots,k_s$, let $R_n[s+1] = R_n[s]$.

\medskip{}

\noindent \textit{Action at stage $s+1 = 3t+3$.} If for some $e < k_s$, an element entered $W_{f(e)}$ at stage $t$, do the following. Otherwise, do nothing. Let $k_{s+1} = e$. Let $\bar{u}$ be the elements of $R_e[s]$ and let $\bar{v}$ be the other elements of $\mc{B}[s]$ which are not in $R_e[s]$. Let $\psi(\bar{x},\bar{y})$ be the conjunction of the atomic diagram of $\mc{B}[s]$, so that $\mc{B}[s] \models \psi(\bar{u},\bar{v})$. Then $\mc{A}_{k_s} \models \psi(j[s](\bar{u}),j[s](\bar{v}))$. Since $j[s](\bar{u}) \in \mc{A}_{e} \prec_1 \mc{A}_{k_s}$, there is a tuple $\bar{a} \in \mc{A}_e$ such that $\mc{A}_e \models \psi(j[s](\bar{u}),\bar{a})$. Then define $R_e[s+1] = R_e[s] \cup \{\bar{v}\}$ and define $j[s+1] \supseteq j[s] \upharpoonright_{R_e[s]}$ to map $\bar{v}$ to $\bar{a}$. For $n < e$, define $R_n[s+1] = R_n[s]$.

\medskip{}

Note that at every stage $s$, $j[s](R_n) \subseteq \mc{A}_n$.

\medskip{}

\noindent \textit{End construction.}

\bigskip{}

Let $k = \liminf_s k_s$. If $n \in S$, then $k$ is the least $e$ such that $W_{f(e)}$ is infinite. Otherwise, if $n \notin S$, then $k = \infty$.

\begin{claim}
Fix $n \leq k$. Let $s$ be a stage such that $k_s \geq n$ and after which no element is ever enumerated into $W_{f(e)}$ for any $e < n$. Then:
\begin{enumerate}
	\item for all $t_2 > t_1 \geq s$, $R_n[t_1] \subseteq R_n[t_2]$ and $j[t_1] \upharpoonright R_n[t_1] \subseteq j[t_2]$.
	\item $R_n = \bigcup_{t \geq s} R_n[t]$ is a substructure (in the relational language) of $\mc{B}$.
	\item $j_n = \bigcup_{t \geq s} j[t] \upharpoonright R_n[t]$ is an isomorphism between $R_n$ and $\mc{A}_n$.
\end{enumerate}
Given $m \leq n \leq k$, $R_m \subseteq R_n$.
\end{claim}
\begin{proof}
(1) is easy to see from the construction. (2) is also clear. For (3) it remains to see that $j_n$ is surjective onto $\mc{A}_n$. If $a \in \mc{A}_n$ is the least element which is not in the image of $j$, then there is some stage $t \geq s$ at which each lesser element of $\mc{A}_n$ is already in the image of $j[t]$, and $a$ is among the first $t$ elements of $\mc{A}_n$. For each lesser element $a'$ of $\mc{A}_n$, $a' = j[3t+1](b')$ for some $b'$, and $b' \in R_n[3t+1]$; hence $j[t'](b') = a'$ at each later stage $t' \geq 3t + 1$. Then at some stage, say, $3t + 4$, we put $a$ into the image of $j$, say with $j(b) = a$, and we have $b \in R_n[3t+4]$, so that $j[t'](b) = a$ at each later stage $t' \geq 3t+4$. This is a contradiction; thus $j$ contains all of $\mc{A}_n$ in its image.
\end{proof}

\begin{claim}
$\mc{B} = \bigcup_{n \leq k} R_n$.
\end{claim}
\begin{proof}
If an  element enters $W_{f(e)}$ at stage $t$, and no element ever enters $W_{f(e')}$, for $e' < e$, after stage $t$, then $\mc{B}[3t+3] = R_e[3t+3] \subseteq R_e$. If $k < \infty$, then there are infinitely many stages $3t+3$ at which $\mc{B}[3t+3] = R_k[3t+3]$, and so $\mc{B} = R_k$. If $k = \infty$, then there is a sequence $(e_1,t_1),(e_2,t_2),(e_3,t_3),\ldots$, with $e_1 < e_2 < e_3 < \cdots$ and $t_1 < t_2 < t_3 < \cdots$, at which $\mc{B}[3t_i + 3] \subseteq R_{e_i}[3t_i + 3] \subseteq R_{e_i}$. Then $\mc{B} = \bigcup_{n \leq k} R_n$.
\end{proof}

\begin{claim}
If $m \in S$, then $\mc{B}_m \cong \mc{A}$.
\end{claim}
\begin{proof}
We have $k < \infty$. Then $\mc{B}_m = \bigcup_{n \leq k} R_n = R_k$, and $R_k$ is isomorphic to $\mc{A}$ via $j_k$.
\end{proof}

\begin{claim}
If $m \notin S$, then $\mc{B}_m$ is not finitely generated.
\end{claim}
\begin{proof}
Fix a tuple $\bar{g} \in \mc{B}_m$. Then $\bar{g} \in R_n$ for some $n$. Pick $a \in \mc{A}_{n+1} \setminus \mc{A}_n$. Since $a \notin \mc{A}_n$, $a \notin j(R_n)$. Thus there is $h \in R_{n+1} \setminus R_n$ with $j(h) = a$. Thus $R_n$ is a proper substructure of $\mc{B}$. Since $\bar{g} \in R_n$, $\bar{g}$ cannot generate $\mc{B}$.
\end{proof}

\noindent This completes the proof of the theorem.
\end{proof}

\begin{proof}[Proof of (1)$\Rightarrow$(2) in Theorem \ref{main theorem}]
Let $\mc{A}$ be a finitely generated self-reflective structure which has a d-$\Sigma^0_2$ Scott sentence. Let $X \geq_T \mc{A}$ be such that this Scott sentence is $X$-computable. Then by Theorem \ref{completeness}, the index set $I^X(\mc{A})$ is $\Sigma^0_3(X)$ m-complete relative to $X$, contradicting that $I^X(\mc{A})$ is in d-$\Sigma^0_2(X)$.
\end{proof}

\section{Finitely generated fields}\label{finitely generated fields}

It is not hard to show that every finitely-generated field is self-reflective, and hence has a d-$\Sigma^0_2$ Scott sentence.

\begin{proof}[Proof of Theorem \ref{thm:field}]
Let $F$ be a finitely generated field of characteristic $p$ which is possibly zero. We claim that $F$ is not self-reflective, and hence by Theorem \ref{thm:not-refl}, $F$ has a d-$\Sigma^0_2$ Scott sentence.

Let $\mathbb{F}_p$ be the prime field of characteristic $p$. Write $F = \mathbb{F}_p(a_1,\ldots,a_m,b_1,\ldots,b_n)$, with $a_1,\ldots,a_m$ a transcendence basis for $F$ over $\mathbb{F}_p$, and let $\varphi \colon F \to E \subsetneq F$ be an isomorphism between $F$ and a proper subfield $E$ of $F$. We claim that $E$ is not a $\Sigma^0_1$-elementary substructure of $F$.

Let $a_1',\ldots,a_m'$ be the images of $a_1,\ldots,a_m$ under $\varphi$, and let $b_1',\ldots,b_n'$ be the images of $b_1,\ldots,b_n$ under $\varphi$. Since $F$ and $E = \mathbb{F}_p(\bar{a}',\bar{b}')$ are isomorphic, $a_1',\ldots,a_m'$ are a transcendence base for $E$, and so $\bar{a},\bar{b}$ are algebraic over $\mathbb{F}_p(\bar{a}',\bar{b}')$. Thus the atomic type $\tp_{\text{at}}(\bar{a},\bar{b} / \mathbb{F}_p(\bar{a}',\bar{b}'))$ is isolated by a formula $\psi(\bar{a}',\bar{b}',\bar{x},\bar{y})$. We claim that there is no tuple $\bar{c},\bar{d} \in E$ with $E \models \psi(\bar{a}',\bar{b}',\bar{c},\bar{d})$. Suppose to the contrary that there was such a tuple $\bar{c},\bar{d}$; then $E(c,d)$ would be isomorphic to $F$ over $E$; but since $\bar{c},\bar{d} \in E$, $E(c,d) = E$, and so $E$ is isomorphic to $F$ over $E$. This cannot happen as $E$ is a proper subfield of $F$. This is a contradiction; thus $E$ is not a $\Sigma^0_1$-elementary substructure of $F$, proving the theorem.
\end{proof}

\section{Finitely generated groups}\label{finitely generated groups}

In this section, we first introduce some group theory background on HNN extensions (Section \ref{HNN extension}) and small cancellation theory (Section \ref{small cancellation}). Then we will use this machinery to construct a self-reflective group in Section \ref{self-reflective group}. We refer the interested reader to \cite[\S IV, \S V]{Ly77} for more details on the group theoretic tools we are using here.

\subsection{HNN Extensions}\label{HNN extension}

\begin{definition}
For a group $G$ with presentation $G = \langle S \mid R \rangle$ and an isomorphism $\alpha: H \to K$ between two subgroups $H,K \subseteq G$, we define the \emph{HNN extension} of $G$ by $\alpha$ to be
$$G*_\alpha = \langle S, t \mid R, \{ tht^{-1} = \alpha(h) \}_{h \in H} \rangle.$$
\end{definition}

The key lemma about HNN extensions we will need is the following, which says every trivial word in the HNN extension is either already trivial in $G$, or ``reducible'' by a conjugation of $t$ or $t^{-1}$.

\begin{lemma}[Britton's Lemma]
With the notation above, let \[w = g_0 t^{e_1} g_1 t^{e_2} \cdots t^{e_n} g_n \in G*_\alpha\] with $g_i\in G$, and $e_i = \pm 1$. Suppose $w = 1$, then one of the following is true:
\begin{enumerate}
\item $n =0$ and $g_0 = 1$ in $G$,
\item there is $k$ such that $e_k = 1$, $e_{k+1} = -1$, and $g_k \in H$, or
\item there is $k$ such that $e_k = -1$, $e_{k+1} = 1$, and $g_k \in K$.
\end{enumerate}
\end{lemma}

\noindent One can show using Britton's Lemma that the natural homomorphism from $G$ to $G*_\alpha$ is an embedding, so that we can think of $G$ as a subgroup of $G*_\alpha$.

\subsection{Small Cancellation}\label{small cancellation}

\begin{definition}
We say a presentation $\langle S \mid R \rangle$ is \emph{symmetrized} if every relation is cyclically reduced and the relation set $R$ is closed under inverse and cyclic permutation.

Let $\langle S \mid R \rangle$ be a symmetrized presentation. We say a word $u \in F(S)$ is a \emph{piece} if there are two $r_1 \neq r_2 \in R$ such that $u$ is an initial subword of both $r_1$ and $r_2$. We also say the presentation satisfies the \emph{$C'(\lambda)$ small cancellation hypothesis} if for every relation $r$ and every piece $u$ with $r = uv$, we have $|u| < \lambda |r|$.
\end{definition}

Furthermore, we shall say a non-symmetrized presentation satisfies the small cancellation hypothesis if it does once we replace the relation set with its symmetrized closure. We shall also say a group is a small cancellation when it is clear which presentation we are using.

The key lemma we will need for small cancellation groups is the following, which says that every presentation of the trivial word must contain a long common subword with a relator.

\begin{lemma}[Greendlinger's Lemma]
Let $G = \langle S \mid R \rangle$ be a $C'(\lambda)$ small cancellation group with $0 \leq \lambda \leq \frac16$. Let $w$ be a non-trivial freely reduced word representing the trivial element of $G$. Then there is a cyclic permutation $r$ of a relation in $R$ or its inverse with $r = uv$ such that $u$ is a subword of $w$, and $|u| > (1-3\lambda)|r|$.
\end{lemma}

We say that a word $w$ is \emph{Dehn-minimal} if it does not contain any subword $v$ that is also a subword of a relator $r = vu$ such that $|v| > |r| / 2$. Greendlinger's lemma implies that, given a $C'(1/6)$ presentation of a group, we can solve the word problem using the following observation: a Dehn-minimal word is equivalent to the identity if and only if it is the trivial word. Given a word $w$, we replace $w$ by equivalent words of shorter length until we have replaced $w$ by a Dehn-minimal word $w'$. Then $w$ is equivalent to the identity if and only if $w'$ is the trivial word. This is the Dehn's algorithm.

\subsection{A self-reflective group}\label{self-reflective group}

Let $T$ be the tree (directed acyclic graph) with vertex set $V(T) = \{ (n,\tau) \mid n \in \omega\text{ and }\tau \in \mathbb{Z}^{<\omega}\}$. The parent $(n,\tau)^-$ of $(n,\tau)$ is $(n,\tau^-)$ if $\tau \neq \la \ra$, and $(n+1,\la \ra)$ otherwise. See Figure \ref{treefigure}.

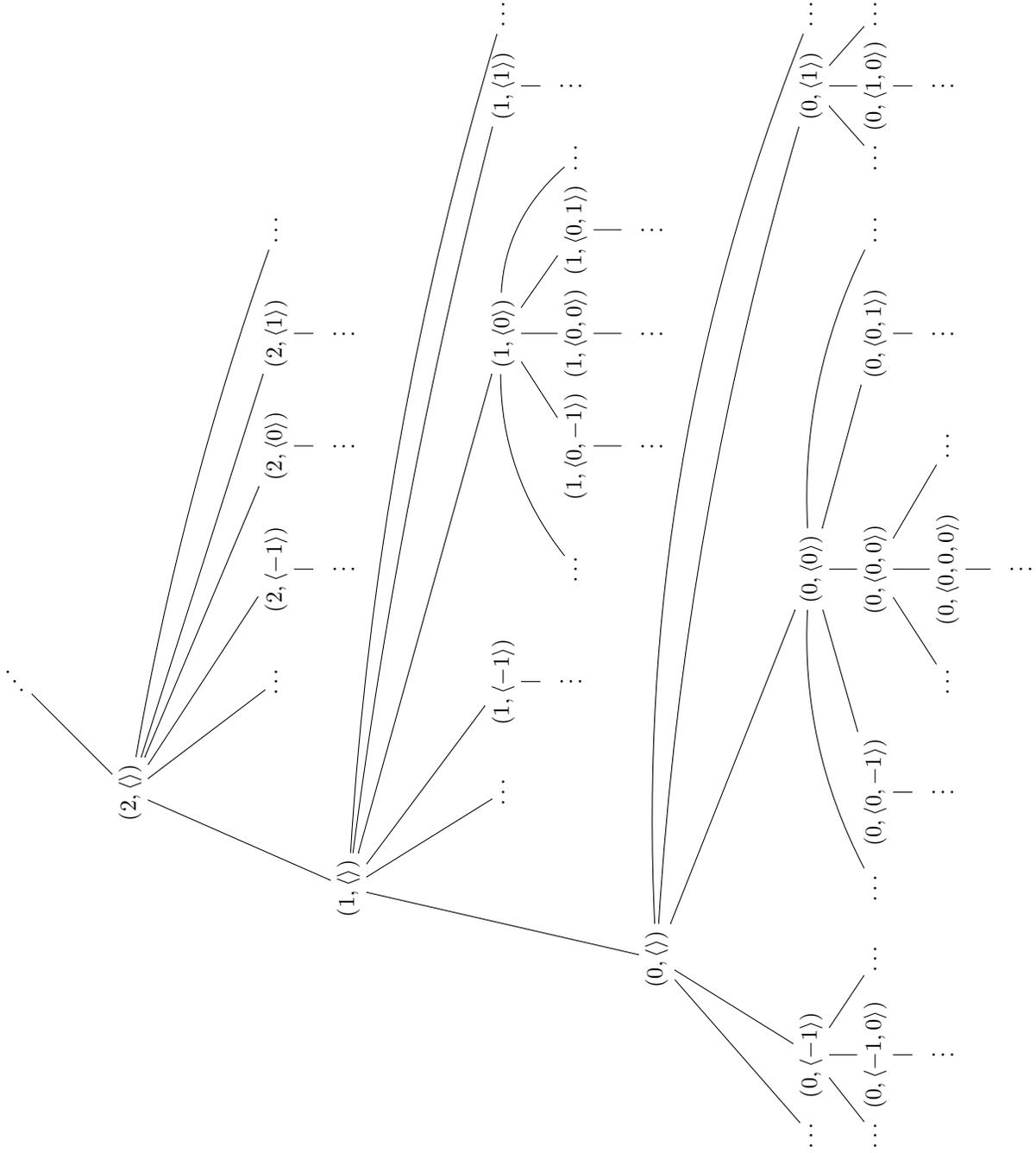
\begin{sidewaysfigure}[pt]
\[ \xymatrix @C=0pc @R=1pc{
&&&&&\iddots \ar@{-}[ddl]\\\\
&&&&(2,\la \ra)\ar@{-}[ddddl]\ar@{-}[dddr]\ar@{-}[dddrr]\ar@{-}[dddrrr]\ar@{-}[dddrrrr]\ar@{-}@/^/[dddrrrrr]\\\\\\
&&&&&\cdots & (2,\la -1 \ra)\ar@{-}[d] & (2,\la 0 \ra)\ar@{-}[d] & (2,\la 1 \ra)\ar@{-}[d] & \cdots
\\
&&&(1,\la \ra)\ar@{-}[dddddl]\ar@{-}[dddr]\ar@{-}[dddrr]\ar@{-}[dddrrrrr]\ar@{-}@/^/[dddrrrrrrrr]\ar@{-}@/^1pc/[dddrrrrrrrrr]&&&\vdots&\vdots&\vdots
\\\\\\
&&&&\cdots & (1,\la -1 \ra)\ar@{-}[d] &&& (1,\la 0 \ra)\ar@{-}[d]\ar@{-}[dl]\ar@{-}@/_1pc/[dll]\ar@{-}[dr]\ar@{-}@/^1pc/[drr] &&& (1,\la 1 \ra)\ar@{-}[d] & \cdots
\\
&&&&&\vdots & \cdots & (1,\la 0,-1 \ra)\ar@{-}[d] & (1,\la 0,0 \ra)\ar@{-}[d] & (1,\la 0,1 \ra)\ar@{-}[d] & \cdots & \vdots
\\
&&(0,\la \ra)\ar@{-}[dddll]\ar@{-}[dddl]\ar@{-}[dddrrrr]\ar@{-}@/^1pc/[dddrrrrrrrrr]\ar@{-}@/^2pc/[dddrrrrrrrrrr]&& & & & \vdots & \vdots & \vdots &&&
 \\\\\\
\cdots & (0,\la -1 \ra)\ar@{-}[d]\ar@{-}[dr]\ar@{-}[dl] &&&&& (0,\la 0 \ra)\ar@{-}[d]\ar@{-}[dll]\ar@{-}@/_1pc/[dlll]\ar@{-}[drr]\ar@{-}@/^1pc/[drrr] &&&&& (0,\la 1 \ra)\ar@{-}[d]\ar@{-}[dr]\ar@{-}[dl] & \cdots
\\
\cdots & (0,\la -1,0 \ra)\ar@{-}[d] & \cdots & \cdots & (0,\la 0,-1 \ra)\ar@{-}[d] && (0,\la 0,0 \ra)\ar@{-}[d]\ar@{-}[dr]\ar@{-}[dl] && (0,\la 0,1 \ra)\ar@{-}[d] & \cdots & \cdots & (0,\la 1,0 \ra)\ar@{-}[d] & \cdots
\\
& \vdots & & & \vdots & \cdots & (0,\la 0,0,0 \ra)\ar@{-}[d] & \cdots & \vdots &&& \vdots
\\
&&&&&& \vdots
} \]
\caption{The tree $T$.}\label{treefigure}
\end{sidewaysfigure}

Let $u(x,y) = xyx^2y\cdots x^{100}y$ be a word in $F(x,y)$. Let $K$ be the group on generators $V(T)\cup\{a\} \cup B$ (where $B = \{ b_i \mid i \in \mathbb{Z}\}$) with relations:
\begin{itemize}
\item $u((n,\tau),a) = (n,\tau)^-$ for every $(n,\tau) \in T$.
\item $u((n,\tau),b_i) = (n,\tau\concat\la i\ra)$ for every $(n,\tau) \in T$ and $i\in\mathbb{Z}$.
\end{itemize}
Note that $K$ is generated by $(0,\la\ra), a$, and $B$: we can generate any vertex $(n,\la \ra)$ by $(1,\la \ra) = u((0,\la\ra),a)$, $(2,\la \ra) = u((1,\la\ra),a)$, and so on, and then we can generate, for example, $(2,\la 5,3 \ra)$, as $(2,\la 5,3 \ra) = u(u((2,\la \ra),b_5),b_3)$. Also note that $K$ is a $C'(\frac1{10})$ small cancellation group. Noting that any reduced word in $B$ is Dehn-minimal, we see that $B$ freely generates a free subgroup of $K$.

\begin{claim}\label{minimality of B}
Let $v$ be a word in $V(T), a, B$, such that $v$ is Dehn-minimal. Then $w$ is in the subgroup $F(B)$ of $K$ generated by $B$ if and only if $v$ is a word in $B$.
\end{claim}

\begin{proof}
The if direction is obvious. For the only if direction, assume we have a word $v$, in $V(T)$, $a$, and $B$, which is equal to a reduced word $v'$ in $B$. If $v'$ was the trivial word, then since $v$ is Dehn-minimal, $v$ would also be the trivial word. So we may assume that $v'$ is not the trivial world. Also, we may assume without loss of generality that $v$ and $v'$ have no common prefix, so that $v^{-1} v'$ is a reduced word. Then, by applying Greendlinger's lemma to $v^{-1}v'$, we get a subword $u$ of $v^{-1}v'$ which is also a subword of a relator $r$, with $|u| > (\frac{7}{10})|r|$. Noting that none of the relators of $K$ has two consecutive $b_i$'s, we see that the subword $u$ of $v^{-1} v'$ given by Greendlinger's lemma has to be contained in $v^{-1}$ except possibly the last letter of $u$. If $u'$ is the part of $u$ which is contained in $v^{-1}$, we have $|u'| \geq |u| - 1 > \frac{1}{2}|r|$ as $|r| > 100$. This contradicts the Dehn-minimality of $v$.
\end{proof}

Now let $G$ be the HNN extension $\la K, t \mid t b_i t^{-1} = \alpha(b_i) = b_{i+1}\ra$ of $K$ by the partial isomorphism $\alpha(b_i) = b_{i+1}$. $G$ is then finitely-generated by $(0,\la\ra), a, b_0,$ and $t$.

\begin{lemma}\label{SR group}
$G$ is self-reflective.
\end{lemma}

\begin{proof}
Let $H \subseteq G$ be the subgroup generated by $(1,\la\ra),a,b_0$, and $t$. We claim that $H$ is a proper $\Sigma^0_1$-elementary subgroup of $G$ which is isomorphic to $G$.

\begin{claim}
$H$ is isomorphic to $G$.
\end{claim}
\begin{proof}
Define the homomorphism $\iota:G \to H \subseteq G$ given by sending $(n,\tau)$ to $(n+1,\tau)$ and fixing $a$, $b_i$, and $t$. We must check that this does indeed define a homomorphism:
\begin{itemize}
	\item $u(\iota(n,\tau),\iota(a)) = u((n+1,\tau),a) = (n+1,\tau)^- = \iota((n,\tau)^{-})$ for every $(n,\tau) \in T$.
	\item $u(\iota(n,\tau),\iota(b_i)) = u((n+1,\tau),b_i) = (n + 1,\tau\concat\la i\ra) = \iota((n,\tau\concat\la i\ra))$
	\item $\iota(t) \iota(b_i) \iota(t)^{-1} = t b_i t^{-1} = b_{i+1} = \iota(b_{i+1})$
\end{itemize}
Since $\iota$ maps relators of $G$ to relators of $G$, it defines a homomorphism.

Now we will check that $\iota$ is an embedding. Suppose $\iota(v) = 1$ for some word $v$ in $V(T), a, B, t$. Without loss, we may assume $v\neq 1$ is a word of minimum length among the words representing the same group element. By abusing notation, we will use $\iota(v)$ to denote the word obtained by replacing each $(n,\sigma)$ in $v$ by $(n+1,\sigma)$; this is a word that represents the group element $\iota(v)$.

Now since $\iota(v) = 1$, by Britton's lemma, either $\iota(v)$ does not contain $t, t^{-1}$, or it contains a subword $tut^{-1}$ or $t^{-1}ut$ with $u\in F(B)$. We claim that we must be in the first case, where $\iota(v)$ (and hence $v$) does not contain $t$ or $t^{-1}$. In the second case, if $\iota(v)$ does contain a subword $tut^{-1}$ or $t^{-1}ut$ with $u\in F(B)$, we can write $u = \iota(u')$, where $t u' t^{-1}$ or $t^{-1} u' t$ appears as a subword of $v$ as $\iota$ leaves $t$ unchanged. If we can show that $u$ is Dehn-minimal and so, by Lemma \ref{minimality of B}, $u$ is actually a word in $B$, then, as $\iota$ leaves $B$ unchanged, $u'$ would also be a word in $B$. By conjugating each $b_i$ in $u'$ by $t$ (or $t^{-1}$) to get $b_{i+1}$ (or $b_{i-1}$), we get a shorter word representing the same element, contradicting the minimality of $v$. We will now argue that $u$ is Dehn-minimal. If $u$ was not Dehn-minimal, this would be witnessed by a subword $w$ of a relator $r$, with $|w| > \frac{1}{2} |v|$. Then looking at all of relators of $K$, we see that $w = \iota(w')$ and $r = \iota(r')$ where $w'$ is a subword of a relator $r'$ of $K$, and also a subword of $u'$. Thus $v$ is not of minimal length, a contradiction. So $u$ is Dehn-minimal, and so $\iota(v)$ does not contain $t, t^{-1}$.

Since $\iota(v)=1$ and contains only $V(T)$, $a$, and $B$, by Greendlinger's lemma, $\iota(v)$ is not Dehn-minimal. However, since any relator $r$ that holds on $\iota(V(T))$, $a$, and $B$ is the image, under $\iota$, of a relator that holds on $V(T)$, $a$, and $B$, this shows that $v$ is also not Dehn-minimal, a contradiction. 
\end{proof}

\begin{claim}
$H$ is a proper subgroup of $G$.
\end{claim}

\begin{proof}
We will show that $(0,\la\ra)\notin H$. Suppose $(0,\la\ra)\in H$. Choose a shortest spelling $v$ of $(0,\la\ra)$ in $\iota(V(T)),a,B, t$. By applying Britton's lemma to $(0,\la\ra)^{-1}v$ and using the same argument as above, we see that $v$ does not contain $t$. Thus, we may apply Greendlinger's lemma on $(0,\la\ra)^{-1}v$ to get a subword that is also a subword of some relator $r$ with length more than half of the length of $r$. However, this subword can not contain $(0,\la\ra)^{-1}$, as the only relation containing $(0,\la\ra)^{-1}$ but not $(0,\sigma)$ for any $\sigma \neq \la\ra$ is $u((0,\la\ra),a) = (1,\la\ra)$, but any long subword of it will contain more than one instance of $(0,\la\ra)$. Thus, the subword must be strictly in $v$, and contradicts the minimality of $v$.
\end{proof}

\begin{claim}
$H$ is a $\Sigma^0_1$-elementary subgroup of $G$.
\end{claim}
\begin{proof}
We only need to show that for every tuple $\overline{g}\in G$, and every quantifier free formula $\psi(\overline{x},\overline{y})$ such that $G\models \psi((1,\la\ra), a, b_0, t, \overline{g})$, there is a tuple $\overline{h}\in H$ such that $H\models \psi((1,\la\ra), a, b_0, t, \overline{h})$. It suffices to show the (stronger) statement that for every finite subset $1\notin S\subset G$, there is a (retraction) $\kappa:G\to G$ such that $\kappa\vert_H = \operatorname{id}_H$, $\kappa(G) = H$, and $1\notin \kappa(S)$. Fixing a shortest spelling in $V(T),a,B,t$ for each element in $S$, we define $\kappa$ by fixing the generators of $H$ and sending $(0,\la\ra)$ to $(1,\la n\ra )$ for $n$ sufficiently large relative to the (length and subscripts of) spelling of elements of $S$.

Suppose there is $s\in S$ with $\kappa(s) = 1$. Write $s$ in the shortest spelling fixed above. We spell $\kappa(s)$ by replacing every $(0,\tau)$ in the shortest spelling of $s$ by $(1,n\concat\tau)$. By Britton's lemma, either there is no $t$ in $\kappa(s)$, or there is a subword $tvt^{-1}$ or $t^{-1}vt$ with $v\in F(B)$. In the second case, by minimality of $s$ and Claim \ref{minimality of B}, we see that $v$ only contains letters $b_i$'s, and thus we can reduce the length of $s$ by replacing each $b_i$ by $b_{i+1}$ (or $b_{i-1}$) to get a shorter spelling of $s$, a contradiction. Thus, $s$ does not have any $t$ in it.

Now, applying Greendlinger's lemma to $\kappa(s)$, we get a subword of $\kappa(s)$ that can be replaced by a shorter string. We will argue that a corresponding replacement can also be carried out for $s$, possibly with a different relator, contradicting the minimality of $s$. We divide into three cases, depending on which relator is used. First, not that the replacement cannot be given by any relator involving $b_m$ for $|m| \geq n$ since $n \gg 1$ implies $s$ does not contain the letter $b_m$ in it; thus the following three cases exhaust the possibilities.

\begin{case}
The relator is $u((1,\la i\ra \concat\sigma),a) = (1,\la i\ra \concat \sigma)^-$ for $|i| \geq n$.
\end{case}

\noindent Since $n \gg 1$, each instance of $(1,\la i\ra \concat\sigma)$ in $\kappa(s)$ comes from an instance of $(0,\sigma)$ in $s$, and each instance of $(1,\la i\ra \concat \sigma)^-$ comes from an instance of $(0,\sigma)^-$ in $s$. (It is important here that $(0,\la \ra)^- = (1,\la \ra) = (1,\la i\ra)^-$.) Thus we can perform a replacement in $s$ using the relator $u((0,\sigma),a) = (0,\sigma)^-$.

\begin{case}
The relator is $w((1,\la i \ra \concat \sigma),b_k) = (1,\la i \ra \concat\sigma\concat\la k\ra)$ for $|i| \geq n$.
\end{case}

\noindent Since $i \gg 1$, each instance of $(1,\la i \ra \concat \sigma)$ in $\kappa(s)$ comes from an instance of $(0,\sigma)$ in $s$, and each instance of $(1,\la i \ra \concat\sigma\concat\la k\ra)$ in $\kappa(s)$ comes form an instance of $(0,\sigma \concat \la k \ra)$ in $s$. Thus we can perform a replacement in $s$ using $w((0,\sigma),b_k) = (0,\sigma\concat\la k\ra)$.

\begin{case}
The relator does not involve any letters $(1,\la i \ra \concat \sigma)$ with $|i| \geq n$.
\end{case}

\noindent In this case, we can apply exactly the same relator to $s$.
\end{proof}

\noindent Thus we have shown that $G$ contains a copy $H$ of itself as a $\Sigma^0_1$-elementary subgroup, and hence is self-reflective.
\end{proof}

\begin{proposition}
$G$ is computable.
\end{proposition}

\begin{proof}
We use the following algorithm to solve the word problem in $G$: for any string in $V(T), a, B, t$, we search and replace the following three types of subwords:
\begin{enumerate}
\item $tvt^{-1}$ with $v$ containing only $b_i$'s. Replace such subwords by deleting $t$ and $t^{-1}$ and replacing each $b_i$ by $b_{i-1}$.
\item $t^{-1}vt$ with $v$ containing only $b_i$'s. Replace such subwords by deleting $t^{-1}$ and $t$ and replacing each $b_i$ by $b_{i+1}$.
\item Subword $v$ such that $v$ is also a subword of a relator $r$ and $|v| > \frac12|r|$. Replace such subwords by the rest of the relator $r$ after deleting $v$.
\end{enumerate}
Since any word can only mention finitely many letters, there are only finitely many possible relators for case (3). Thus, even though we have infinitely many relators, the search in (3) is still finite. Since these replacements shorten the length of the word, for any input word, sequences of such replacements terminate. If the resulting word is trivial, we output ``The input word is equal to the identity.'', otherwise output ``The input word does not equal the identity.''

To verify this algorithm is valid, consider a word that represents the identity, on which the algorithms terminates with a non-trivial word $v$. Since we can not perform any more replacement of the third kind, $v$ is Dehn-minimal. Thus, by Lemma \ref{minimality of B} and Britton's lemma, we should be able to do a replacement of either the first or the second kind, a contradiction.
\end{proof}

\section{Finitely generated rings}\label{finitely generated rings}

In this section, we use the group ring construction to produce a ring that is self-reflective. Notice that the group ring $R[G]$ is computable if both $G$ and $R$ are.

\begin{theorem}
Let $G$ be the self-reflective group defined in Section \ref{finitely generated groups}. Then the group ring $R[G] = \{f:G\to R \mid |\operatorname{supp}(f)| < \infty \}$ over any finitely generated ring $R$ is also self-reflective.
\end{theorem}

\begin{proof}
Note that any endomorphism $\alpha$ of $G$ induces an endomorphism $\alpha^*$ of $R[G]$ by pre-composition and fixing $R$. Furthermore, if the endomorphism on $G$ is injective, then the induced endomorphism on $R[G]$ is also injective.

Let $\iota$ be as defined in Lemma \ref{SR group}. Then $\iota^*$, the induced endomorphism of $R[G]$, is also injective and not surjective. Call $B = \iota^*(R[G])$. Note that $B$ is just $R[H]$, where $H = \iota(G)$.

Now, as in Lemma \ref{SR group}, it suffices to show that for every finite subset $0 \notin T \subset R[G]$, there is a retraction $\beta :R[G] \to R[G]$ with $\beta(R[G]) = B$, $\beta\vert_{B} = \operatorname{id}\vert_{B}$, and $0 \notin \beta(T)$. Let $U$ be all the group elements that appear in some members of $T$, and $S = \{ u_1u_2^{-1} \mid u_1 \neq u_2\in U\}$. Since $1 \notin S$, the proof of Lemma \ref{SR group} gives a retraction $\kappa: G \to G$ such that $1 \notin \kappa(S)$. Now the induced endomorphism $\beta = \kappa^*$ is also a retraction. Furthermore, if $\kappa^*(t) = 0$ for some $t\in T$, since $1\notin \kappa(S)$, $\kappa^*$ is injective on the support of $t$, thus we must have $t = 0$, a contradiction. Thus $R[G]$ is also self-reflective.
\end{proof}

\bibliography{SSoG}
\bibliographystyle{alpha}

\end{document}